\documentclass[11pt]{amsart}
\usepackage{amssymb, amsthm, epsfig, verbatim, amscd, enumerate, stmaryrd,
amsmath}

\newtheorem{thm}{Theorem}[section]
\newtheorem{cor}[thm]{Corollary}
\newtheorem{lem}[thm]{Lemma}

\theoremstyle{definition}
\newtheorem{defn}[thm]{Definition}

\newtheorem{rem}[thm]{Remark}
\newtheorem{que}[thm]{Question}

\theoremstyle{remark}

\numberwithin{equation}{section}

\newcommand{\al}{\alpha}
\newcommand{\be}{\beta}
\newcommand{\ga}{\gamma}

\newcommand{\de}{\delta}
\newcommand{\ep}{\varepsilon}

\newcommand{\la}{\lambda}

\newcommand{\si}{\sigma}
\newcommand{\Si}{\Sigma}

\newcommand{\Legknot}{{\mathcal{L}}}
\newcommand{\Frontproj}{{\overline {\mathcal{L}}}}
\newcommand{\x}{\times}

\newcommand{\R}{\mathbb R}

\newcommand{\del}{\partial}

\newcommand{\co}{\colon\thinspace} 

\newcommand{\rank}{\operatorname{rank}}

\begin{document}
\mathsurround=1pt 
\title[Maps on $3$-manifolds]
{Maps on $3$-manifolds given by surgery}



\author{Boldizs\'ar Kalm\'ar}
\address{Alfr\'ed R\'enyi Institute of Mathematics,
Hungarian Academy of Sciences \newline
Re\'altanoda u. 13-15, 1053 Budapest, Hungary}
\email{bkalmar@renyi.hu}

\author{Andr\'as I. Stipsicz}
\address{Alfr\'ed R\'enyi Institute of Mathematics,
Hungarian Academy of Sciences \newline
Re\'altanoda u. 13-15, 1053 Budapest, Hungary and \newline
Institute for Advanced Study, Princeton, NJ, 08540}
\email{stipsicz@renyi.hu}


\begin{abstract}
Suppose that the $3$-manifold $M$ is given by integral surgery along a
link $L\subset S^3$.  In the following we construct a stable map from
$M$ to the plane, whose singular set is canonically oriented.  We
obtain upper bounds for the minimal numbers of crossings and non-simple
singularities and of connected components of fibers of stable maps
from $M$ to the plane in terms of properties of $L$.
\end{abstract}

\maketitle

\section{Introduction}
\label{sec:intro}
It is well-known that a continuous map between smooth manifolds can be approximated by a smooth map and any smooth map on a 3-manifold can be approximated by a generic stable  map.
This line of argument, however, gives no concrete map on a given 
3-manifold $M$ even if  it is given by some explicit construction. 
Recall that 
by \cite{Lic0, Wal} a closed oriented 3-manifold  $M$ can be given by integral surgery
along some link  $L$ in $S^3$. 
In the present work we construct an explicit  stable map $F \co M \to \R^2$ based on such a 
surgery presentation of $M$.

Results in  
\cite{Gr09, Gr10} give lower bounds  
 on the topological complexity of the set of critical 
values  
of generic smooth maps and
on the complexity of the fibers
 in terms of the topology of the source and target manifolds.
In a slightly different direction,
\cite{CoTh} gives a lower bound for the number of crossing singularities 
of stable maps from a $3$-manifold to $\R^2$ in terms of the Gromov norm of the $3$-manifold.
Recently \cite{Ba08, Ba09} and \cite{GK07} 
studied the topology of $4$-manifolds through the singularities of their  maps into surfaces.

In the present paper we give upper bounds on the minimal numbers of
the crossings and non-simple singularities and of the connected
components of the fibers of stable maps on the $3$-manifold $M$ in
terms of properties of diagrams of $L$ (e.g.\ the number of crossings
or the number of critical points when projected to $\R$).  As an
additional result, these constructions lead to upper bounds on a
version of the Thurston-Bennequin number of negative Legendrian knots.

Before stating our main results, we need a little preparation.
First of all, a stable map of a $3$-manifold into the plane can be easily 
described by its Stein factorization.
 
\begin{defn}
Let
$F$ be a  map of the $3$-manifold $M$ into $\R^2$.
Let us call two points  $p_1, p_2 \in M$ {equivalent} if and only if 
$p_1$ and $p_2$ lie on the same component of an $F$-fiber.
Let $W_F$ denote the quotient space of $M$ with respect
to this equivalence relation and $q_F \co M \to W_F$ the quotient map.
Then there exists a unique continuous map $\Bar{F} \co W_F \to \R^2$ such that
$F = \Bar{F} \circ q_F$. The space $W_F$ or the factorization of the 
map $F$ into the composition of $q_F$ and $\Bar{F}$ is called the \emph{ Stein
factorization} of the map $F$. (Sometimes the map $\Bar{F}$ is also called the  
{Stein factorization} of  $F$.) 
\end{defn}

In other words, the Stein factorization $W_F$ is the space of
connected components of fibers of $F$.  Its structure is strongly
related to the topology of the $3$-manifold $M$. For example, an
immediate observation is that the quotient map $q_F \co M \to W_F$
induces an epimorphism between the fundamental groups since every loop
in $W_F$ can be lifted to $M$.  If $F \co M \to \R^2$ is a stable map,
then its Stein factorization $W_F$ is a $2$-dimensional CW complex.
The local forms of Stein factorizations of proper stable maps of
orientable $3$-manifolds into surfaces are described in \cite{KuLe,
  Lev}, see Figure~\ref{st2}.  Indeed, let $F$ be a stable map of the
closed orientable $3$-manifold $M$ into $\R^2$.  We say that a
singular point $p \in M$ of $F$ is of type {\sc (a)}, \ldots, {\sc
  (e)}, respectively, if the Stein factorization $\bar F$ at $q_F(p)$
looks locally like (a), \ldots, (e) of Figure~\ref{st2}, respectively.
We will call a point $w \in W_F$ a singular point of type {\sc (a)},
\ldots, {\sc (e)}, respectively, if $w = q_F(p)$ for a singular point
$p \in M$ of type {\sc (a)}, \ldots, {\sc (e)}, respectively.
According to \cite{KuLe, Lev} we give the following characterization
of the singularities of $F$: The singular point $p$ is a {\it cusp}
point if and only if it is of type {\sc (c)}, the singular point $p$
is a {\it definite fold} point if and only if it is of type {\sc (a)}
and $p$ is an {\it indefinite fold} point if and only if it is of type
{\sc (b)}, {\sc (d)} or {\sc (e)}.  Singular points of types {\sc (d)}
and {\sc (e)} are called {\it non-simple}, while the others are called
{\it simple}.  A double point in $\R^2$ of two crossing images of
singular curves which is not an image of a non-simple singularity is
called a {\it simple singularity crossing}.  A {simple singularity
  crossing} or an image in $\R^2$ of a non-simple singularity is
called a {\it crossing singularity}.  A stable map is called a
\emph{fold map} if it has no cusp singularities.

\begin{figure}[ht] 
\begin{center} 
\epsfig{file=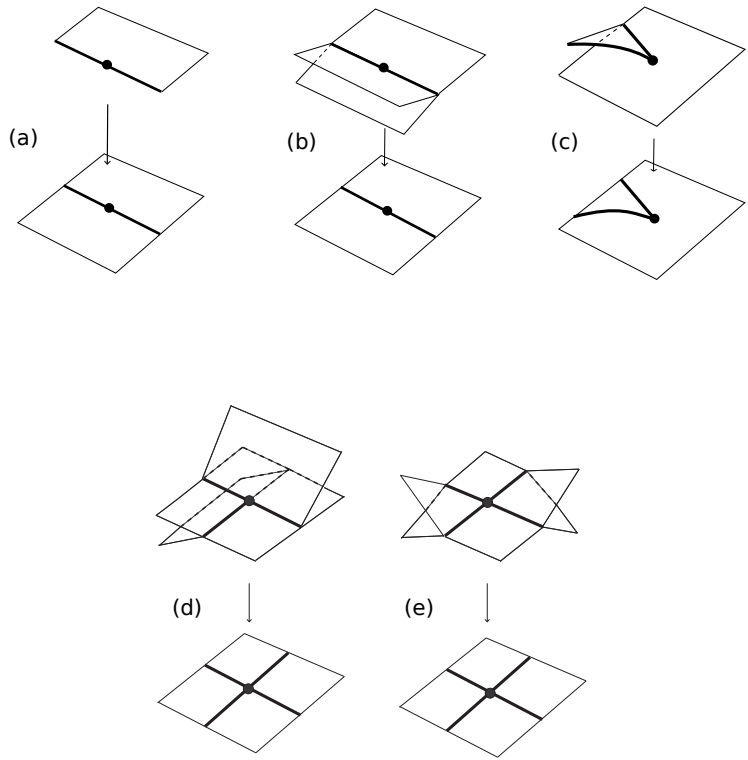, height=11cm} 
\end{center} 
\caption{{\bf The local forms of Stein factorizations of stable maps from orientable $3$-manifolds to surfaces.} The map (symbolized by an arrow) maps from the CW complex $W_F$ to $\R ^2$.}  
\label{st2}  
\end{figure} 

Let $L\subset \R ^3\subset S^3$ be a given link, and let
 $\overline L$ denote a  generic projection of it to the plane.
 Let ${\rm {n}}( L)$ and ${\rm {cr}}(\overline L)$ denote
 the number of components of $L$ and
   the number of crossings of $\overline L$, respectively.  
   
      Choose a direction in $\R^2$, which we represent by a vector $v \in \R^2$.
We can assume that $v$ satisfies the condition that 
the projection of the diagram $\overline L$ to $\R v^{\bot}$ along $v$ 
yields only 
non-degenerate critical points. 
Let 
${\rm {t}}(\overline L) = {\rm {t}}_v(\overline L)$ denote the number of  times 
$\overline L$ is tangent to  $v$.
Suppose at each $v$-tangency $p$ the half line 
 emanating from $p$ in the direction of $v$ avoids the crossings of $\overline L$ and
 intersects $\overline L$ transversally (at the points different from $p$).
 Denote the number of transversal intersections by 
 $\ell(\overline L, v, p)$. 
Let $\ell(\overline L, v)$ denote the maximum of the values $\ell(\overline L, v, p)$, where
$p$ runs over the $v$-tangencies.
With these definitions in place now we can state the main result of the paper.

\begin{thm}\label{fothm}
Suppose that the 3-manifold  $M$
is obtained by integral  surgery on the link $L\subset S^3$.
Then there is a stable map $F \co M \to \R^2$
such that 
\begin{enumerate}[\rm (1)]
\item
the Stein factorization $W_F$ is homotopy equivalent 
to the bouquet $\bigvee_{i=1}^{{\rm {n}}( L)} S^2$,
\item
the number of cusps of $F$ is equal to 
${\rm t}_v(  \overline L )$,
\item
all the non-simple singularities of $F$ are
 of type {\rm \sc(d)}, and their number 
  is equal to 
${\rm cr}(  \overline L ) + \frac{3}{2}{\rm t}_v(  \overline L ) -{\rm n}(L)$,
\item
the number 
of non-simple singularities which are not connected by any singular arc of type {\rm \sc(b)}
to any cusp is equal to
${\rm cr}(  \overline L ) + \frac{1}{2}{\rm t}_v(  \overline L ) -{\rm n}(L)$,
\item
the number of simple singularity crossings of $F$ in $\R^2$ is no more than
$$8{\rm cr}(  \overline L )  + 6\ell(\overline L, v){\rm t}_v(  \overline L ) + {\rm t}_v(  \overline L )^2,$$
\item
the number of connected components of the singular set of $F$ is no more than ${\rm {n}}( L) + \frac{3}{2}{\rm t}_v(\overline L) +1$, and
\item
the maximal number of the connected components of any fiber of $F$ is no more than 
${\rm t}_v(  \overline L ) + 3$.
\item 
Suppose we got $M$ by cutting out and gluing back 
the regular neighborhood $N_L$ of $L$ from $S^3$.
Then
the indefinite fold singular set of $F$ 
contains a link 
in $S^3 - N_L$,
which is isotopic to $L$ in $S^3$
and whose $F$-image coincides with 
$\overline L$. 
\end{enumerate} 
\end{thm}

\begin{rem}\noindent
\begin{enumerate}
\item
Let $Y$ be a closed orientable $3$-manifold, $f$ a given smooth map of
$Y$ into $\R^2$ and $L \subset Y$ a link disjoint from the singular
set of $f$.  Suppose furthermore that $f|_L$ is an immersion.  Let $M$
denote the $3$-manifold obtained by some integral surgery along $L$.
Then the method developed in the proof of Theorem~\ref{fothm} provides
a stable map of $M$ into $\R^2$ (relative to $f$).
\item
In constructing the map $F$, the proof of Theorem~\ref{fothm} provides
a sequence of stable maps $f_0, f_1, \ldots, f_6$ of $S^3$ into
$\R^2$, where each $f_i$ is obtained from $f_{i-1}$ by some
deformation, $i=1, \ldots, 6$.  Finally, the map $F$ is obtained from
$f_6$.  Suppose that $X$ is a compact $4$-manifold which admits a handle
decomposition with only $0$- and $2$-handles, i.e.  $X$ can be given
by attaching 4-dimensional 2-handles to $D^4$ along $S^3$.  Using our
method we can construct a stable map $G$ of $X$ into $\R^2 \x [0,1]$.
\end{enumerate}
\end{rem}

Recall that according to \cite{BuRh} a closed orientable $3$-manifold
$M$ has a stable map into $\R^2$ without singularities of types {\sc
  (b)}, {\sc (c)}, {\sc (d)} and {\sc (e)} if and only if $M$ is a
connected sum of finitely many copies of $S^1 \x S^2$.  According to
\cite{Sa} a closed orientable $3$-manifold $M$ has a stable map into
$\R^2$ without singular points of types {\sc (c)}, {\sc (d)} and {\sc
  (e)} if and only if $M$ is a graph manifold.  By \cite{Lev0} a
$3$-manifold always has a stable map into $\R^2$ without singular
points of type {\sc (c)}.  Our arguments imply a constructive proof
for

\begin{thm}\label{foldcsak(e)}
Every closed orientable  $3$-manifold
has a stable map into $\R^2$
without singular points of types {\rm \sc (c)} and {\rm \sc (e)}. 
\end{thm}

\begin{rem}\noindent
\begin{enumerate}
\item
One cannot expect to eliminate the singular points of types {\sc (a)},
{\sc (b)} or {\sc (d)} of stable maps from arbitrary closed orientable
$3$-manifolds to $\R^2$.  In this sense our Theorem~\ref{foldcsak(e)}
gives the best possible elimination on $3$-manifolds.
 \item
By taking an embedding $\R^2 \subset S^2$ we get for every closed
orientable $3$-manifold a stable map into $S^2$ as well without
singular points of types {\rm \sc (c)} and {\rm \sc (e)}.  Then by
using the method of \cite{Sa06}, for example, for eliminating the
singular points of type {\rm \sc (a)}, we get a stable map, which is a
direct analogue of the indefinite generic maps appearing in
\cite{Ba08, Ba09, GK07}.
\end{enumerate}
\end{rem}

The construction also implies certain relations between quantities one can naturally 
associate to stable maps and to surgery diagrams.

\begin{defn}\label{stabmapquant}
{Suppose that $M$ is a fixed closed, oriented $3$-manifold and $F \co
  M \to \R^2$ is a stable map with singular set $\Si$.
\begin{itemize} 
\item Let ${\rm s}( F )$ denote the number of simple singularity
  crossings of $F$.
\item
Let ${\rm ns}( F )$ denote the  number of non-simple singularities of $F$.
\item
Let ${\rm d}(F)$ denote the number of crossing singularities of $F$.
Clearly ${\rm s}( F ) + {\rm ns}( F ) = {\rm d}(F)$.
\item
Let ${\rm nsnc}( F )$ denote
the number 
of non-simple singularities of $F$ which are not connected by any singular arc
of type {\sc  (b)}
to any cusp.
\item
Let ${\rm c}(F)$ 
denote the number of cusps of $F$. Clearly ${\rm nsnc}( F ) + {\rm c}(F) \geq {\rm ns}( F )$.
\item
Let ${\rm cc}(F)$ denote the  number of  connected components of $F(\Si)$.
Clearly it is no more than the number of connected components of $\Si$.
\item Let ${\rm cf}(F)$ denote the  maximum number of  connected components of the fibers of $F$.
\end{itemize}}
\end{defn}

The inequality  
\[
\rank H_*(M) \leq 2{\rm d}(F) + {\rm c}(F) + 2{\rm cc}(F)
\]
has been shown to hold in \cite[Section~2.1]{Gr09}.\footnote{The paper \cite{Sa95} is also closely related.}
 In addition, by \cite[Theorem~3.38]{CoTh}  
we have ${\rm d}(F) \geq || M || / 10$, 
where $|| M ||$ is the Gromov norm of $M$, cf. also \cite[Section 3]{Gr09}.

Theorem~\ref{fothm} provides several estimates for upper bounds on the
topological complexity of smooth maps of a $3$-manifold given by
surgery. For example, by summing quantities in
Definiton~\ref{stabmapquant} and their estimates in
Theorem~\ref{fothm}, we immediately obtain

\begin{cor}
Suppose that the 3-manifold  $M$
is obtained by integral  surgery on the link $L \subset S^3$. 
Let $\overline L$ be any diagram of $L$ and $v$ a general position vector in $\R^2$.
Then
\begin{itemize}
\item
$\min {\rm d}(F) \leq 9{\rm cr}(\overline L) + (6\ell(\overline L,
  v)+\frac{3}{2}){\rm t}_v( \overline L ) + {\rm t}_v( \overline L )^2 -{\rm
    n}(L)$,
\item
$\min {\rm cf}(F) \leq  {\rm t}_v(  \overline L ) +3$,
\item
$\min \{ 2{\rm d}(F) + {\rm c}(F) + 2{\rm cc}(F) \} \leq 18{\rm
  cr}(\overline L) + (12\ell(\overline L, v)+7){\rm t}_v( \overline L ) + 2{\rm t}_v(
  \overline L )^2 + 2$,
\end{itemize}
where the minima are taken for all the stable maps $F$ of $M$ into
$\R^2$.  Evidently, we can estimate other properties in
Definiton~\ref{stabmapquant} of stable maps on $M$ as well.
\end{cor}

These expressions can be simplified by estimating $\ell(\overline L, v)$ as
\begin{equation}\label{eq:egysz}
\ell(\overline L, v) \leq {\rm t}_v(  \overline L ) -1
\end{equation}
cf. Lemma~\ref{egyszerusit}.

The number of tangencies of a projection of a knot in a fixed
direction is reminiscent to the number of cusp singularities of a
front projection of a Legendrian knot in the standard contact
3-space. Based on this analogy, our previous results imply an estimate
on a quantity attached to a Legendrian knot in the following way.

Recall first that the standard contact structure $\xi _{st}$ on $\R^3
$ is the 2-plane field given by the kernel of the 1-form $\alpha =dz
+xdy$. A knot $\Legknot$ is \emph{Legendrian} if the tangent vectors of $\Legknot$
are in $\xi _{st}$. (To indicate the Legendrian structure on the knot, we will denote it
by $\Legknot$ and reserve the notation $L$ for smooth knots and links.)
If $\Legknot$ is chosen generically within its Legendrian
isotopy class, its projection to the $(y,z)$ plane will have no
vertical tangencies, and at any crossing the strand with smaller slope
will be over the one with higher slope. Consider now a Legendrian knot
$\Legknot $ and let $\Frontproj$ denote such a projection (called a
\emph{front} projection) of $\Legknot$. The \emph{Thurston-Bennequin number}
${\rm {tb}}(\Legknot )$ of $\Legknot $ is given by the formula $w(\Frontproj
 )-\frac{1}{2}(\# cusps(\Frontproj ))$, where $w(\Frontproj)$
stands for the \emph{writhe} (i.e.\ the signed sum of the
double points) of the projection. Although the definition of tb$(\Legknot )$
uses a projection of the Legendrian knot $\Legknot$, it is not hard to show that 
the resulting number is an invariant of the Legendrian isotopy class of $\Legknot$.

 In case the projection has only
negative crossings, we have that $w(\Frontproj)= -{\rm
  {cr}}({\overline {L}})$, hence the resulting Thurston-Bennequin
number can be identified with $- {\rm cr}( \overline L ) - \frac{1}{2} 
{\rm   {t}}_v(\overline L)$ after choosing $v$ appropriately, cf.\ \cite{Geiges,
  OS}. (In this case the generic projection ${\overline {L}}$ used in the definitions of 
${\rm   {t}}_v(\overline L)$ and 
${\rm  {cr}}({\overline {L}})$ is derived from the front projection $\Frontproj$ by rounding the cusps.)

As it is customary, we define ${\rm {TB}}(L)$ as the maximum of all Thurston-Bennequin
numbers of Legendrian knots smoothly isotopic to  $L$. 
(It is a nontrivial fact, and follows
from the tightness of $\xi _{st}$ that this maximum exists.) A modification of this
definition for negative knots (i.e.\ for knots admitting projections with only negative crossings) provides

\begin{defn}
For a negative knot $L \subset \R^3$
let ${\rm TB}^-( L )$ denote
the value $\max \{  {\rm tb} (\Legknot)   \}$ where
$\Legknot$ runs over those   Legendrian knots smoothly isotopic to $L$ 
which admit front diagrams  with only negative crossings.
\end{defn}
\noindent It is rather easy to see that if the knot $L$ admits a projection 
with only negative crossings, then it also has a front projection with the same property.
Clearly ${\rm TB}^-( L ) \leq {\rm TB}( L )$.

\begin{thm}\label{nagyTBbecsles}
For a negative knot $L \subset \R^3$ and any
$3$-manifold $M$ obtained by an integral surgery along $L$
we have 
\begin{itemize}
\item
${\rm TB}^-( L ) \leq - \min   \frac{\sqrt{{\rm s}( F )}}{2\sqrt{7}}$,
\item
${\rm TB}^-( L ) \leq - \min  \frac{\sqrt{{\rm d}( F )}}{2\sqrt{7}}$,
\item
${\rm TB}^-( L ) \leq - \min  {\rm nsnc}( F ) - 1$,
\end{itemize}
where the minima are taken for all the stable maps $F$ of $M$ into $\R^2$.
\end{thm}

By Theorem~\ref{nagyTBbecsles} and 
\cite[Theorem~3.38]{CoTh} we obtain

\begin{cor}
For a negative knot $L \subset \R^3$ and any
$3$-manifold $M$ obtained by an integral surgery along $L$, we have
${\rm TB}^-( L ) \leq - \frac{\sqrt{|| M ||}}{2\sqrt{70}}$.
\end{cor}

\bigskip

{\bf Acknowledgements:} 
The authors were supported by OTKA NK81203 and by the
\emph{Lend\"ulet program} of the Hungarian Academy of Sciences.
The first author was  partially supported by Magyary Zolt\'an Postdoctoral Fellowship.
The authors thank the anonymous referee for the comments which improved the paper.

\section{Preliminaries} 

In this section, we recall and summarize  some technical tools.
First, we show
 that a cusp can be pushed through an indefinite fold arc as in Figure~\ref{movingcusp}:

\begin{lem}[Moving cusps]\label{cuspmovelem}
Suppose that in a neighborhood $U$ of a point $p \in M$ the Stein factorization of a map $f \co M \to \R^2$
is given by  Figure~\ref{movingcusp}(a).
Then $f$ can be deformed in this neighborhood to a map $f'$  so that
the Stein factorization of $f'$ is as the diagram of 
Figure~\ref{movingcusp}(b).
\end{lem}
\begin{proof}
Suppose $q \in M$ is the cusp singular point 
and $\al \subset M$ is the indefinite fold arc 
at hand. 
Let $x \in \R^2$ be a point
on the other side of $f(\al)$ in $f(U)$.
Connect $f(q)$ and $x$ by an embedded arc $\be'$.
Then there is an arc $\be \subset M$ 
such that
$f(\be) = \be'$,
$\be$ starts at $q$
 and $\be$ and $\al$ do not intersect.
By using the technique of \cite{Lev0} we can now deform $f$ in a small tubular neighborhood of $\be$ to achieve the
claimed map $f'$.
Note that during this move one singular point of type {\sc  (d)} appears.
\end{proof}

An analogous statement holds if we move a cusp from a $1$-sheeted region to a $2$-sheeted region.

\begin{figure}[ht] 
\begin{center} 
\epsfig{file=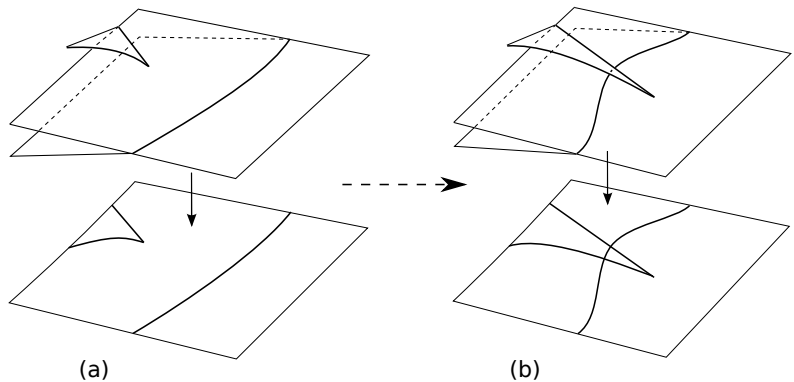, height=5cm} 
\end{center} 
\caption{{\bf Moving cusps.}
A map can be deformed so that the image of a cusp point
goes to the other side of the image of an indefinite fold arc.}
\label{movingcusp}  
\end{figure}

According to the next result, two cusps can be eliminated as in Figure~\ref{elimcusp}:
\begin{lem}[Eliminating cusps]\label{cuspelimlem}
Suppose that in a neighborhood $U$ of a point $p \in M$ the Stein factorization of a map $f \co M \to \R^2$
is given by  Figure~\ref{elimcusp}(a).
Then $f$ can be deformed in this neighborhood to a map $f'$  so that
the Stein factorization of $f'$ 
 is as the diagram of
   Figure~\ref{elimcusp}(b).
\end{lem}
\begin{proof}
This statement is  the elimination in  \cite[pages 285--295]{Lev0} for
$3$-dimensional source manifolds. 
\end{proof}

\begin{figure}[ht] 
\begin{center} 
\epsfig{file=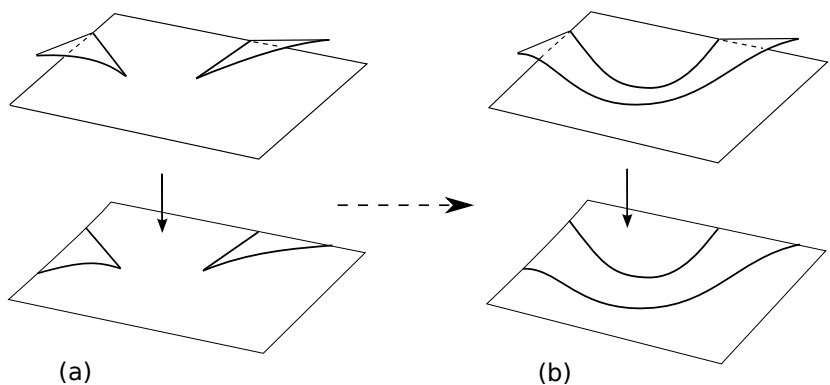, height=5cm} 
\end{center} 
\caption{{\bf Eliminating cusps.}}  
\label{elimcusp}  
\end{figure}

Recall that if $f \co M \to \R^2$ is a stable map and
$S_f \subset M$ denotes its singular set, then $f|_{S_f}$
is a generic immersion with cusps, i.e.\ 
if $C_f \subset M$ denotes the set of cusp points, then
$f|_{S_f  - C_f}$ is a generic immersion with finitely many double points and
$f|_{C_f}$ is disjoint from $f|_{S_f  - C_f}$.

The following result will be the key ingredient in our subsequent arguments for
proving Theorem~\ref{fothm}.
\begin{lem}[Making wrinkles]\label{wrink}
Suppose that $f \co M \to \R^2$ is a stable map
 and
let $L\subset M$ denote 
an embedded  closed $1$-dimensional manifold 
 such that $L$ is disjoint from
the singular set $S_f$,  $f|_{L}$ is a generic immersion 
and $f|_{L \cup S_f}$ is a generic immersion with cusps.
Let $N_{L}$ be a small tubular neighborhood of $L$ disjoint from $S_f$ and
fix an identification of $N_L$
 with the normal bundle of $L$. 
Let $s \co L \to N_{L}$ be a non-zero section 
such that $f(s(x)) \neq f(x)$ for any $x \in L$.
Then $f$ is homotopic to a smooth stable map $f'$ such that
\begin{enumerate}[\rm (1)]
\item 
$f = f'$ outside $N_{L}$,
\item
the singular set of $f'$ is $S_f \cup L \cup s(L)$,
\item
$f'$ has indefinite fold singularities along  $L$,
\item
$f'$ has definite fold singularities along  $s(L)$,
\item
$f'|_{L} = f|_{L}$,
\item
$f'|_{s(L)}$ is an immersion parallel to $f|_{L}$ and
\item
if for a double point $y$ of $f|_{L}$ the two points in 
$f^{-1}(y) \cap L$ lie in the same connected component of the fiber $f^{-1}(y)$, then 
 the double point $y$ of $f'|_L$ correspond to a
 singularity of type {\sc  (d)}.
\end{enumerate}
\end{lem}
\begin{proof}
We perform the homotopy inside $N_{L}$ fiberwise as shown by Figure~\ref{homot}.
Since $N_L$ is the trivial bundle, the homotopy of the fibers yields a homotopy of the 
entire $N_L$.
\end{proof}

\begin{figure}[ht] 
\begin{center} 
\epsfig{file=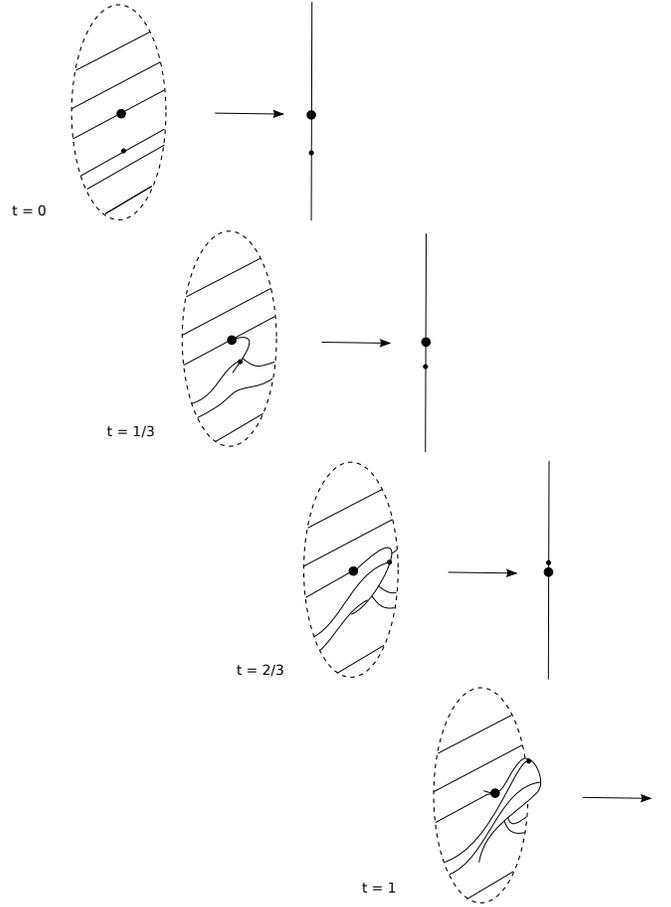, height=12cm} 
\end{center} 
\caption{{\bf The deformation of $f$ to $f'$ in a fiber of $N_{L}$.}}  
\label{homot}  
\end{figure}

\begin{rem}\label{cuspwrink}
If the submanifold $L$ has boundary, we can still get something similar.
In this case the section $s$ should be  zero at the boundary points of $L$, and
the homotopy yields a stable map $f'$ with cusps at $\del L$.
\end{rem}

\section{Proof of the results}

\subsection{Construction of the stable map on $M$}

\begin{proof}[Proof of Theorem~\ref{fothm}]
We will prove the theorem by presenting an algorithm which produces
the map $F$ on $M$ with the desired properties. This algorithm will be
given in seven steps; the first six of these steps are concerned with
maps on $S^3$. Let us start with a fold map $f_0 \co S^3 \to \R^2$
with one unknotted circle $C \subset S^3$ as singular set such that
$f_0\vert _C$ is an embedding and $f_0^{-1}(p)$ is a circle for each
regular point $p \in f_0(S^3)$.  Then the Stein factorization of $f_0$
is a disk together with its embedding into $\R^2$.  By cutting out the
interior of a sufficiently small tubular neighborhood $N_C$ of $C$
from $S^3$, we get a solid torus $S^3 - {\rm {int}}\,N_C$ whose
boundary is mapped into $\R^2$ by $f_0$ as a circle fibration over a
circle parallel to $f_0(C)$, and $f_{0}|_{S^3 -  {\rm {int}}\,N_C}$ is a trivial
circle bundle $S^1 \x D^2 \to D^2$.  Suppose the link $ L \subset S^3$
is disjoint from $N_C \cup \{1\} \x D^2$. Then by identifying $S^3 -
(N_C \cup \{1\} \x D^2)$ with $\R^3$ and $f_{0}|_{S^3 - (N_C \cup
  \{1\} \x D^2)}$ with the projection onto $\R^2$, we get a link
diagram $\overline L = f_0( L)$. Now we start modifying this map $f_0$. In
Steps 1 through 6 we will deal with maps on $S^3$, and the goal will
be to obtain a map which is suitable with respect to the fixed surgery
link $L$.  In particular, we aim to find a map on $S^3$ with the
property that its restriction to any component of $L$ is an embedding
into $\R ^2$.  We suppose that the modifications through Step 1,
\ldots, Step 6 happen so that all the images of the maps $f_1$,
\ldots, $f_6$ lie completely inside the disk determined by the
(unchanged) circle $f_i(C)$, $i = 1, \ldots, 6$. This can be reached
easily by choosing $f_0(C)$ to bound an area ``large enough'' in
$\R^2$ and supposing that the diameter of $\overline L$ is small.

\subsection*{Step 1}
Our first goal is to deform $f_0$ so that the resulting map $f_1$ has
 fold singularities along $L$.
Apply Lemma~\ref{wrink} to the map $f_0 \co S^3 \to \R^2$ 
and the embedded $1$-dimensional manifold $ L \subset S^3$, and 
denote the resulting stable map by $f_1$. It is a fold map,
its indefinite fold singular set is $ L$ and its definite fold singular set is 
$C \cup  L'$, where $ L' = s(L)$ is isotopic to $ L$; for an example see 
Figure~\ref{elsowrinkled}.

\begin{figure}[ht] 
\begin{center} 
\epsfig{file=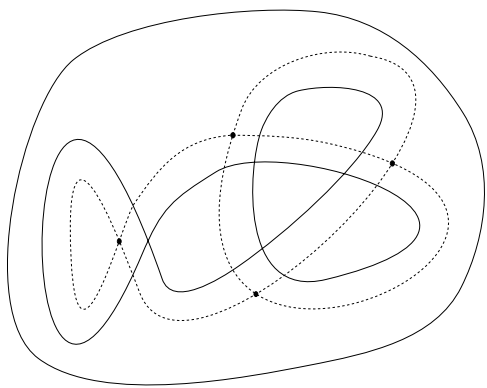, height=5cm} 
\end{center} 
\caption{{\bf The image of the singular set of the map $f_1 \co S^3 \to \R^2$, 
where $ L$ is the trefoil knot.}
The outer circle is $f_1(C)$, the inner solid curve is $f_1( L')$
 and the dashed curve is $f_1( L)$.
The double points of $f_1( L)$ correspond to singularities of type {\sc  (d)}.}  
\label{elsowrinkled}  
\end{figure} 

Since $ L'$ is isotopic to $ L$, 
the integral surgery along $L$ giving $M$ can be equally
performed along $L'$.
Recall that doing surgery along $L'$ simply means 
that we cut out a tubular neighborhood of the definite fold curve $ L'$
(which is diffeomorphic to $ L' \x D^2$),
and glue it back by a diffeomorphism of its boundary $ L' \x S^1$.
If the image $f_1( L')$ was an embedding of circles, then it would be 
easy to construct the claimed map $F$ on the $3$-manifold given by the integral surgery.
Since this is not the case in general, we need to further deform the map $f_1$.

Let 
$B$ denote the interior of the bands (one for each component of $L$)
bounded by $q_{f_1}(L)$ and $q_{f_1}(L')$ in the Stein factorization $W_{f_1}$.
Then $B$ is immersed into $\R^2$ by $\bar f_1$.
The Stein factorizations of the maps $f_2, \ldots, f_6$ in the next steps
will be built on $B$.
Let $B'$ denote the surface $W_{f_1} - {\rm cl}B$.

\subsection*{Step 2}
Now, our goal is to deform $f_1$ so that the Stein factorization of
the resulting map $f_2$ has small ``flappers'' near $q_{f_2}(L')$ at
the points where $\bar f_2( q_{f_2}(L'))$ is tangent to the general
position vector $v$. These ``flappers'' will help us to move the image
of $L$ so that it will become an embedding into $\R ^2$.

First, we use Lemma~\ref{wrink} together with Remark~\ref{cuspwrink}
as follows.  Let $T$ be the set of points in $q_{f_1}(L')$ such that
for each $p \in T$ the direction $v$ is tangent to $f_1(L')$ at $\bar
f_1(p)$.  For each $p \in T$ take a small embedded arc $\al_p$ in a
small neighborhood of $p$ in $B$ such that $\bar f_1\vert _{\al_p}$ is an
embedding parallel to $f_1(L)$.  For each arc $\al_p$ there exists an
embedded arc $\tilde \al_p$ in $S^3$ such that $q_{f_1}|_{\tilde
  \al_p}$ is an embedding onto $\al_p$.  See, for example, the upper
picture of Figure~\ref{kinovesek}, where the small dashed arcs having
cusp endpoints represent the arcs $f_1(\tilde \al_p) =\bar f_1 (\al
_p)$ for all $p \in T$.

\begin{figure}[ht] 
\begin{center} 
\epsfig{file=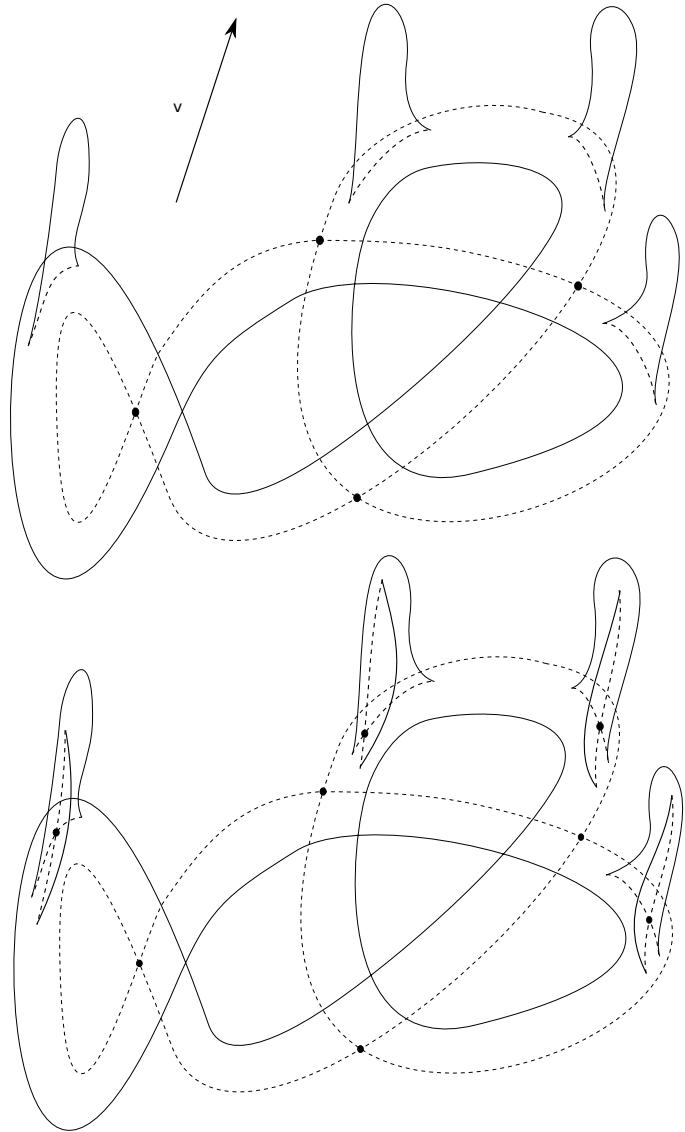, height=15cm} 
\end{center} 
\caption{We obtain the upper picture by applying Lemma~\ref{wrink} and
  Remark~\ref{cuspwrink} to the small arcs $\{ \tilde \al_p : p \in T
  \}$ in $S^3$ which are mapped by $f_1$ to the dashed arcs near the
  points of the diagram $\overline L$ where it is tangent to $v$.  We
  obtain the lower picture by applying Lemma~\ref{wrink} and
  Remark~\ref{cuspwrink} to the new arcs added to the upper
  picture. The solid arcs correspond to singularities of type {\sc
    (a)} and the black double points of the dashed arcs correspond to
  singularities of type {\sc (d)}.}
\label{kinovesek}  
\end{figure}

Apply Lemma~\ref{wrink} and Remark~\ref{cuspwrink}
to the map $f_1\colon S^3\to \R ^2$ and the arcs $\{ \tilde \al_p \subset S^3: p \in T \}$
to obtain a map $f_1'$. 
The section $s$ in Lemma~\ref{wrink} is chosen 
so that if we project
the $f_1'$-images of the arising new definite fold curves in $\R^2$  to $\R v$, then
for each curve there is  only one critical point, which is a maximum.
An example for the resulting map $f_1'$
  can be seen in  the upper  picture of 
Figure~\ref{kinovesek}. 
Note that the deformation yielded small ``flappers'' in $W_{f_1'}$ attached to $B$
along the arcs $\{  \al_p : p \in T \}$. 
Next, for each $p \in T$ take small arcs $\be_p$ in $W_{f_1'}$
which intersect generically the previous arcs $\{  \al_p : p \in T \}$, 
lie in $B$ and on the ``flappers''
and are mapped into $\R^2$ almost parallel to $v$. See  the new small dashed arcs in the lower picture of 
Figure~\ref{kinovesek}.
Once again, there are small arcs $\{ \tilde \be_p : p \in T \}$ embedded in $S^3$ 
mapped by $f_1'$ onto $\{ \be_p : p \in T \}$, respectively. 

The application of 
Lemma~\ref{wrink} and Remark~\ref{cuspwrink}
for these arcs provides us a map, which we denote by $f_2$.
This map will have one additional flapper for every flapper of $f_1'$.
We choose the section $s$ in Lemma~\ref{wrink} so that
the $f_2$-images of the arising new definite fold curves
 lie inward\footnote{At a point of $\{ \bar f_1'(p) : p \in T \}$ let us call the direction 
which is perpendicular to $f_1'(L')$ and points toward 
the direction where locally $f_1'(L')$ lies
``inward''.}
 from the arcs $\{ \bar f_1'(\be_p) : p \in T \}$, respectively,  in the $\bar f_2$-image
  of $B$ and  the previous flappers.
For an enlightening example, see
the  lower picture of 
Figure~\ref{kinovesek}.  
Note that after this step $|T|$ new singular points of type {\sc  (d)} appeared.
Also note that for each $p \in T$ we have four cusp singular points in $S^3$, three of which are mapped by $q_{f_2}$ into $B$. We denote the set of these three cusps  by $C_p$.
For each $p \in T$ the $f_2$-images of two of these three cusps in $C_p$ point to the direction $-v$.
We denote the set of these two cusps  by $D_p$.
Note that the definite fold curves in the images of the two cusps in $D_p$ 
are on opposite sides.

\subsection*{Step 3}

Now our goal is to obtain definite fold arcs connecting points of $S^3$ 
where $f_2$ had cusps.
Moreover these definite fold arcs will be mapped into $\R^2$ parallel to the diagram
$\overline L$. (These curves will be translated in the next step so that later they will lead to an embedding
of $L$ into $\R ^2$.)

In order to reach this goal, we deform the map $f_2 \co S^3 \to \R^2$ further by  eliminating half of the cusps 
as follows.
We proceed for each component of $L$ separately and in the same way, thus in the following
we can suppose that $L$ is connected.
Take a cusp $q_0 \in S^3$ which is in $C_x-D_x$ for an $x \in T$
such that the entire $f_2(L')$ lies to the right hand side of its tangent at $\bar f_2(x)$.
By going along the band $B$ in $W_{f_2}$ in the direction to which  
the $f_2$-image of this cusp $q_0$ points, 
we reach another cusp $q_1$ in $C_p$ for some $p \in T$
at the next $v$-tangency of $f_2(L')$.
If this cusp does not belong to $D_p$, then
it is possible to apply Lemma~\ref{cuspelimlem} and eliminate these two cusps,
since they are in the position of Figure~\ref{elimcusp}.
Then we continue by taking the cusp in $D_p$
whose Stein factorization is folded inward. 
If the cusp $q_1$ does belong to $D_p$, then we choose 
that cusp from $D_p$ which can be used to eliminate $q_0$
(it is easy to see that 
this is exactly the cusp in $D_p$
whose Stein factorization is folded inward), we eliminate them, then
we continue by taking the cusp belonging to $C_p-D_p$.
This procedure goes all along the band $B$, meets all $p \in T$
and eliminates half of the cusps.
After finishing this process, we obtain a stable 
map, which we  denote by $f_3$; cf. Figure~\ref{kinovesekcuspelim} 
for an example.  
\begin{figure}[ht] 
\begin{center} 
\epsfig{file=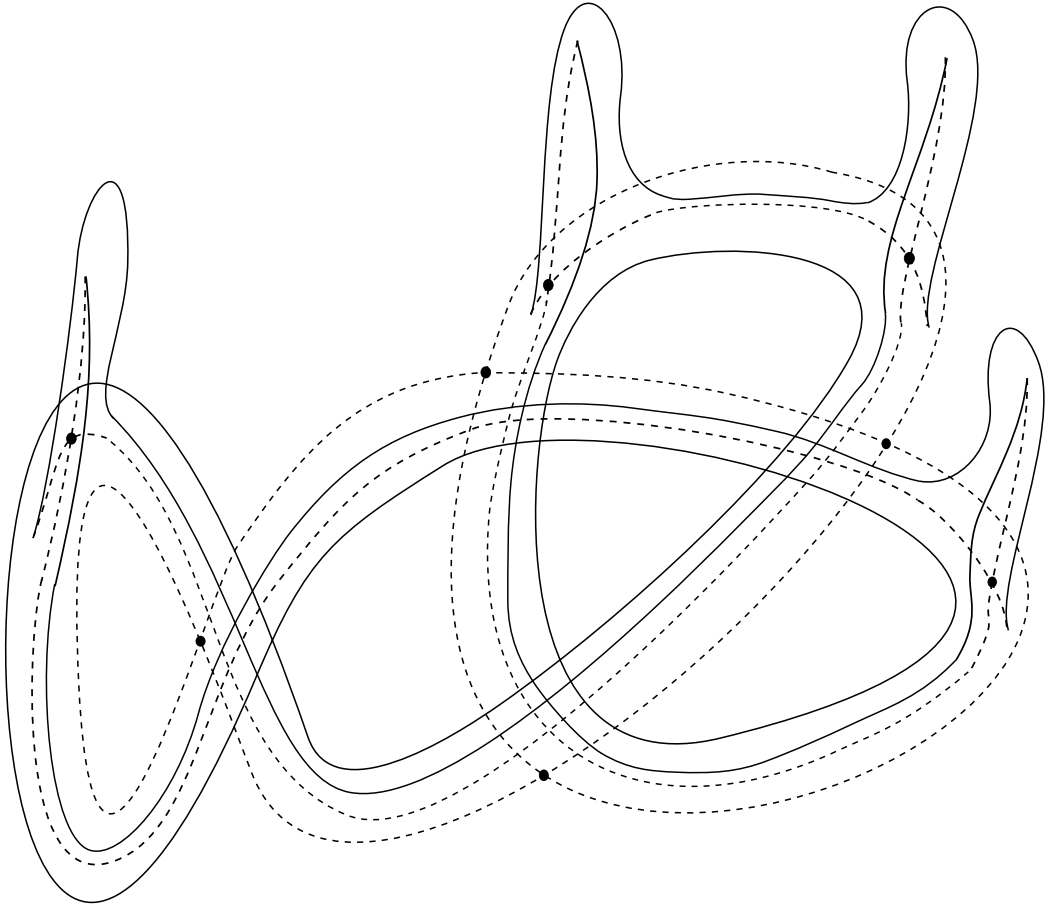, height=12cm} 
\end{center} 
\caption{{\bf Eliminating half of the cusps in the lower picture of Figure~\ref{kinovesek}.} The black double points correspond to singularities of type {\sc (d)}.}  
\label{kinovesekcuspelim}  
\end{figure} 
The cusp elimination results new definite fold curves whose 
$f_3$-image is an immersion, and which have double points
near the crossings of the diagram $\overline L$.
In the next step we will deform $f_3$ so that
the double points of these new curves will be localized near the images of the remaining cusps.

\subsection*{Step 4}
Now our goal is to deform $f_3$ to a map $f_4$ such that the definite
fold arcs obtained in the previous step will be mapped into $\R^2$ far
from the diagram $\overline L$.  (Informally, we will ``lift'' some of the
arcs in the direction of $v$.)  Moreover, the immersion of these
definite fold arcs into $\R^2$ will have double points only near some
cusps of $f_4$. This brings us closer to the original goal to have a
map which embeds a link isotopic to $L$ into the plane.

The cusp eliminations above affect only small tubular neighborhoods of
curves connecting cusps in $S^3$. Denote by $\de \subset S^3$ the new definite
fold arcs which appear in these tubular neighborhoods after the eliminations.
Note that by the algorithm above, the arcs $\de$ are mapped into $\R^2$ so that by an elementary deformation
they can be moved ``upward'' in the direction of $v$, see Figure~\ref{kinovesekcuspelim}.

So we further deform $f_3 \co S^3 \to \R^2$ to get a stable map
denoted by $f_4$ as indicated in
Figure~\ref{kinovesekcuspelimfelnovel}: as it is shown by the picture,
the arcs are ``lifted''.  In fact, we deform $\bar f_3$: we move the
top of the ``flappers'' corresponding to the $\al$-curves of Step 2
and the $\bar f_3$-image of the curves $q_{f_3}(\de)$ in the direction
of $v$ and far from $f_3(L)$.  We proceed for each component of $L$
separately and in the same way, thus in the following we can suppose
that $L$ is connected.  First we choose a point $x \in T$ such that
the entire $f_3(L')$ lies to the right hand side from its tangent at
$\bar f_3(x)$.  Then, by walking along the band $B \subset W_{f_3}$
starting from $x$, we deform the flappers and the curves $\bar f_3(
q_{f_3}(\de))$ to be mapped into the plane as a ``zigzag'' far away
from the diagram $\overline L$.  More precisely, consider the coordinate
system in $\R^2$ with origin $x$ and coordinate axes $\R v^{\perp}$
and $\R v$, respectively, where $v^{\perp}$ denotes the vector
obtained by rotating $v$ clockwise by $90$ degrees.  By extending the
$\bar f_3$-image of the flappers in the direction of $v$ deform the
$\bar f_3$-image of the curves $q_{f_3}(\de)$ so that by going along
$B$ between the points $p_i, p_{i+1} \in T$, where $1 \leq i \leq
|T|-1$ and $p_1 = x$, the corresponding component of the curve
$f_3(\de)$ is mapped into a small tubular neighborhood of a line with
slope $(-1)^{i+1}$ for $i = 1, \ldots, |T|-1$. Finally, arrange the
last component of ${f_3}(\de)$ starting with slope $-1$ and ending at
the first (extended) flapper belonging to $x$, see
Figure~\ref{kinovesekcuspelimfelnovel}.

Note that as a result the double points of the immersion of the deformed curves 
$f_4(\de)$ are in a small neighborhood of
the cusps mapped close to the tops of the flappers.
\begin{figure}[ht] 
\begin{center} 
\epsfig{file=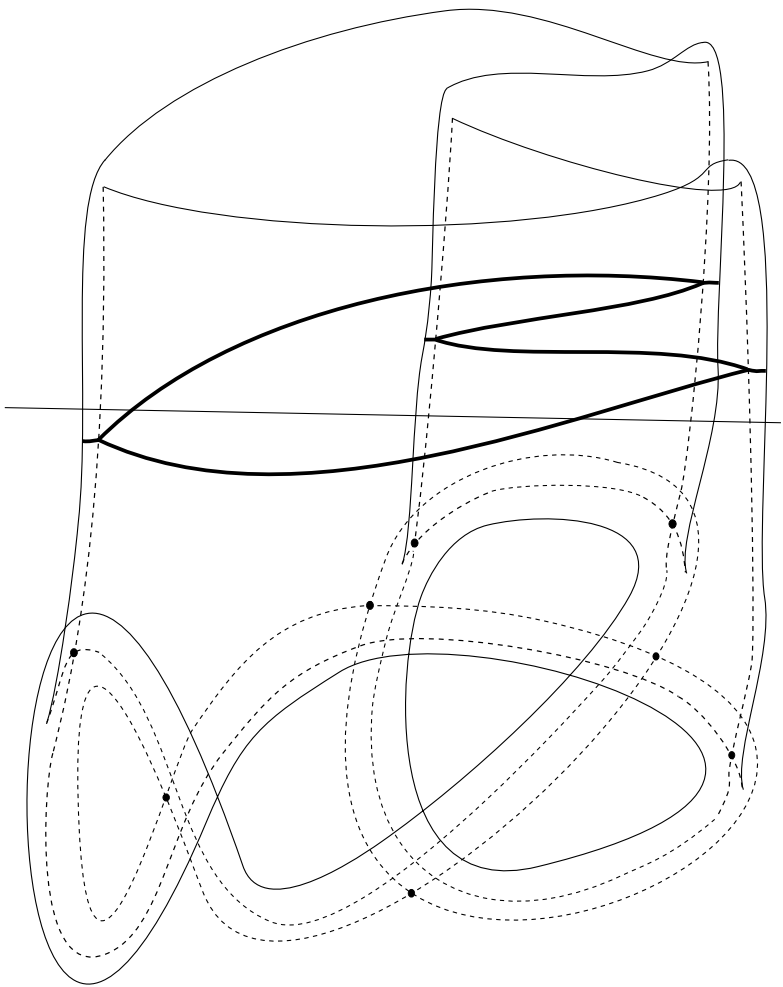, height=13cm} 
\end{center} 
\caption{{\bf The Stein factorization of $f_4$, i.e.\ the deformation of  $f_3$ of Figure~\ref{kinovesekcuspelim}.} 
(The straight  line represents the line $l$ used to cut $W_{f_4}$ in Step 5.)
The upper part of $W_{f_4}$ from the bold $1$-complex is denoted by $A$. 
(As usual, the circle $f_4(C)$ is omitted.)}  
\label{kinovesekcuspelimfelnovel}  
\end{figure} 

\subsection*{Step 5}
In this step, we modify the stable map 
$f_4$ so that the cusps of the resulting map $f_5$ will be easy to eliminate in the next step.
Let $l \subset \R^2$ be a line perpendicular to $v$
located near   $\bar f_4(B)$, separating it from the other parts moved to the direction of $v$ in 
Step~4, as indicated  in Figure~\ref{kinovesekcuspelimfelnovel}.

Now, we cut the $2$-complex  $W_{f_4} - B'$ (recall that $B'$ denotes $W_{f_1} - {\rm cl}B$, see Step 1) along
the $\bar f_4$-preimage of the line $l$, thus
we obtain the decomposition
$$W_{f_4} = A \cup_{\bar f_4^{-1}(l) \cap (W_{f_4} - B')} A',$$
where $A'$ denotes the $2$-dimensional CW complex 
containing $q_{f_4}(L)$ and $A$ denotes the closure of $W_{f_4} - A'$.
Then $q_{f_4}^{-1}(A)$ is a $3$-manifold with boundary. Let us denote 
 the $1$-complex $q_{f_4}(\del q_{f_4}^{-1}(A))$ by $\del A$.
In order to visualize $\del A$
in Figure~\ref{kinovesekcuspelimfelnovel},
we suppose that the cutting of $W_{f_4}$ 
along $\bar f_4^{-1}(l) \cap (W_{f_4} - B')$ is a little bit perturbed and thus
the bold $1$-complex in Figure~\ref{kinovesekcuspelimfelnovel} represents $\del A$. Before proceeding further, we need a better understanding of the 
$q_{f_4}$-preimages of the sets appearing in the above decomposition.
The preimage $q_{f_4}^{-1}(\del A)$ is clearly diffeomorphic to
$J \x S^1$ for a link $J \subset S^3$.
The following statements show much more about $q_{f_4}^{-1}(\del A)$.
 It is  easy to see that
the numbers of components of $J$ and $L$ are equal. However, we have the stronger

\begin{lem}\label{boldlong}
A longitudinal curve in $q_{f_4}^{-1}(\del A)$ is isotopic to $L$. 
\end{lem}
\begin{proof}
The $1$-complex $\del A$ decomposes as a union of $1$-cells: some of
them (which we depict as ``small $1$-cells'' in
Figure~\ref{kinovesekcuspelimfelnovel}) are attached at one of their
endpoints to the union of the other 1-cells, we denote these small
cells by $\si_i$ for $i = 1, \ldots, {|T|}$.  Others are attached by
both of their endpoints.  Let $\si$ denote the $1$-complex $\del A -
\bigcup_{i=1}^{|T|} \si_i$.  Then the PL embedding $\si \subset
W_{f_4}$ is isotopic to the subcomplex $\iota$ of $W_{f_4}$ formed by
the arcs of type {\sc (b)} in the open bands $B$ connecting the
singular points of type {\sc (d)} in $B$.  Furthermore, the subcomplex
$\iota$ is isotopic to $q_{f_4}(L')$.  Take a small closed regular
neighborhood $N$ of $q_{f_4}(L')$.  Then $q_{f_4}^{-1}(N)$ is
naturally a $D^2$-bundle over $L'$.  The boundary of $N$ in $W_{f_4}$
is a $1$-manifold isotopic to $q_{f_4}(L')$, and we will denote it by
$\la$.  Clearly $q_{f_4}^{-1}(\la)$ is diffeomorphic to $L' \x
S^1$. Note that any section of $q_{f_4}^{-1}(\la)$ is isotopic to
$L'$.

The isotopy between $\la$ 
and $\iota$ and the isotopy between 
$\iota$ and $\si$ can be chosen easily so that they
give a PL embedding $\ep \co S^1 \x [0,1] \to W_{f_4}$
such that $S^1 \x \{0\}$ and $S^1 \x \{1\}$
correspond to $\la$ and $\si$, respectively.
For $j = 1, \ldots, |T|$,
let $U_j$ denote small regular neighborhoods of 
the singular points of type {\sc  (d)}
located near the cusp points in $B$
in $W_{f_4}$,
such a $U_j$ and the restriction $\bar f_4|_{U_j}$ can be seen in Figure~\ref{st2}(d). Then 
the intersection $\ep(S^1 \x [0,1]) \cap (\bigcup_{j=1}^{|T|} U_{j})$
consists of a union of disks, which will be  denoted by $\bigcup_{j=1}^{|T|}  D_j$.

First, observe that for each $j = 1, \ldots, |T|$ there exists a disk
$\tilde D_{j}$ embedded into $q_{f_4}^{-1}(U_{j})$ in $S^3$ whose
boundary $\del \tilde D_{j}$ is mapped by $q_{f_4}$ homeomorphically
onto the boundary $ \del D_{j}$, i.e.\ $ \del \tilde D_{j}$ is a
lifting of $\del D_{j}$.  To see this, consider the $3$-manifold
$q_{f_4}^{-1}( U_{j})$ for each $j = 1, \ldots, |T|$.  By \cite{Lev}
the manifold $q_{f_4}^{-1}( U_{j})$ is diffeomorphic to $R \x [0,1]$,
where $R$ is a disk with three holes and it is mapped by ${f_4}$ into
$\R^2$ as we can see in Figure~\ref{nonsimpleanal}(a).  
\begin{figure}[ht]
\begin{center}
\epsfig{file=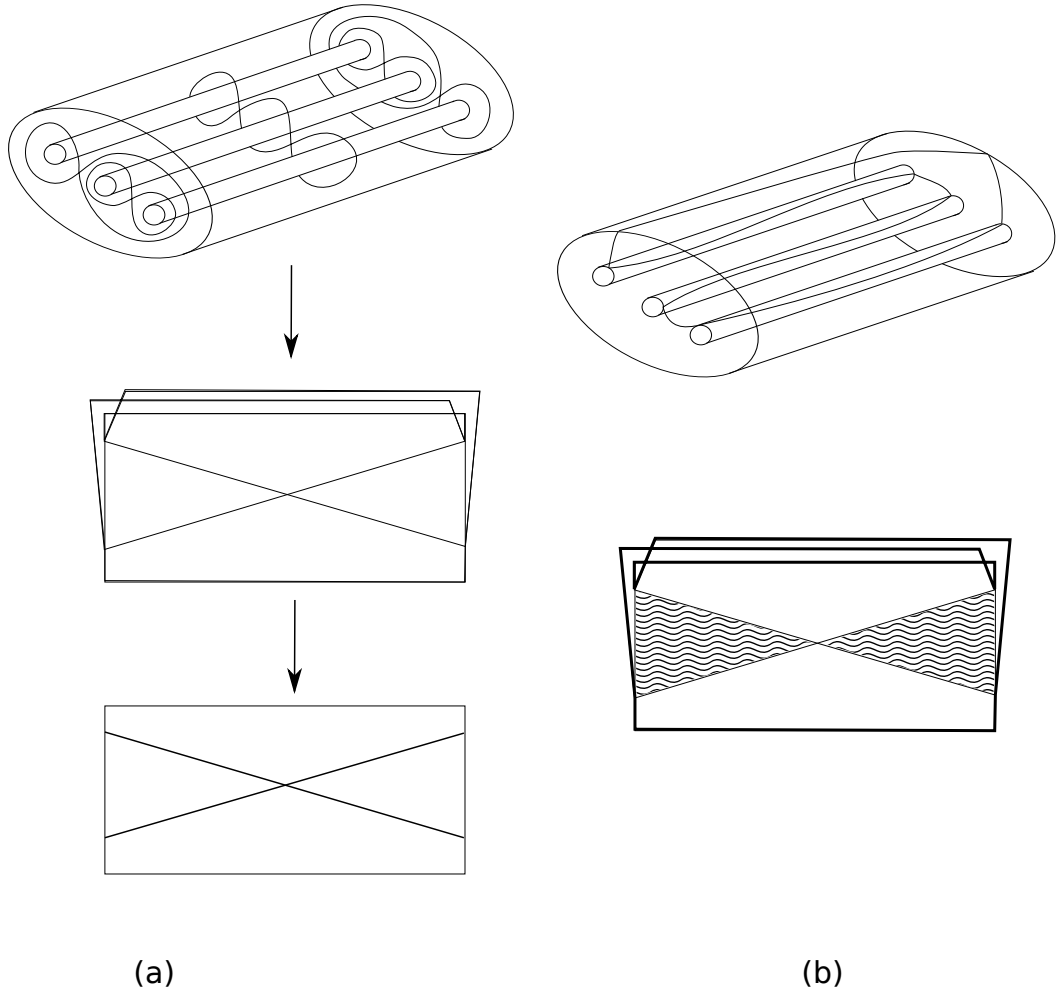, height=13cm}
\end{center}
\caption{In (a) we can see the manifold $R \x [0,1]$ and how it is
  mapped onto the regular neighborhood $U_j$ and into $\R^2$,
  cf.\ Figure~\ref{st2}(d).  $R \x \{0\}$ is mapped onto the left side
  of the rectangle $\bar f_4(U_j)$ as a proper Morse function with two
  indefinite critical points. The two ``figure eights'' in $R \x                    
  \{0\}$ are the two singular fibers.  $R \x \{1\}$ is mapped
  similarly onto the right side of $\bar f_4(U_j)$.  The middle fiber
  in $R \x [0,1]$ is mapped to the singular point of type {\sc (d)}.
  For a detailed analysis see \cite{Lev}. In (b) we can see the
  boundary $\del \tilde D_{j}$ in $R \x [0,1]$ and its image in $U_j$
  represented by a bold $1$-complex.}
\label{nonsimpleanal}
\end{figure}
Each disk $D_j$ can be located in $U_j$ essentially in four ways, for
example the lower picture of Figure~\ref{nonsimpleanal}(b) shows the
disk $D_j$ for the leftmost non-simple singularity crossing of type
{\sc (d)} in Figure~\ref{kinovesekcuspelimfelnovel}. We get $D_j$ on
the picture by cutting out the two shaded areas from the $2$-complex
$U_j$.  It is easy to see in the upper picture of
Figure~\ref{nonsimpleanal}(b) how to put the disk $\tilde D_j$ into $R
\x [0,1]$.  The other three possibilities for the location of a disk
$D_j$ in $U_j$ and the disk $\tilde D_j$ in $q_{f_4}^{-1}( U_{j})$ can
be described in a similar way.

Now observe that $\ep(S^1 \x [0,1]) - \bigcup_{j=1}^{|T|} D_{j}$ can
be lifted to $S^3$ extending $\bigcup_{j=1}^{|T|} \tilde D_{j}$
because of the following.  First, the regular neighborhoods of the
singular points of type {\sc (c)} in $B$ (see Figure~\ref{st2}(c))
intersect $\ep(S^1 \x [0,1])$ in disks which can be lifted to $S^3$.
Then the intersection of the small regular neighborhoods of the
singular curves of type {\sc (b)} and $\ep(S^1 \x [0,1])$ can be
lifted as well since there is no constraint for the lift at the
regular points of $f_4$.  Finally observe that the rest of $\ep(S^1 \x
[0,1])$ intersects $W_{f_4}$ only in areas of non-singular points
which are attached to the boundary of $\ep(S^1 \x [0,1])$, so the
previous lifts extend over the entire $\ep(S^1 \x [0,1])$.

Hence we obtain an embedding $\tilde \ep \co S^1 \x [0,1] \to S^3$
with $S^1 \x \{0\}$ and $S^1 \x \{1\}$ corresponding to lifts of $\la$
and $\si$, respectively. Thus we obtain an isotopy between a longitude
of $q_{f_4}^{-1}(\del A)$ and a lift of $\la$. The fact that any lift
of $\la$ is isotopic to $L'$ finishes the proof.
\end{proof}

\begin{figure}[ht] 
\begin{center} 
\epsfig{file=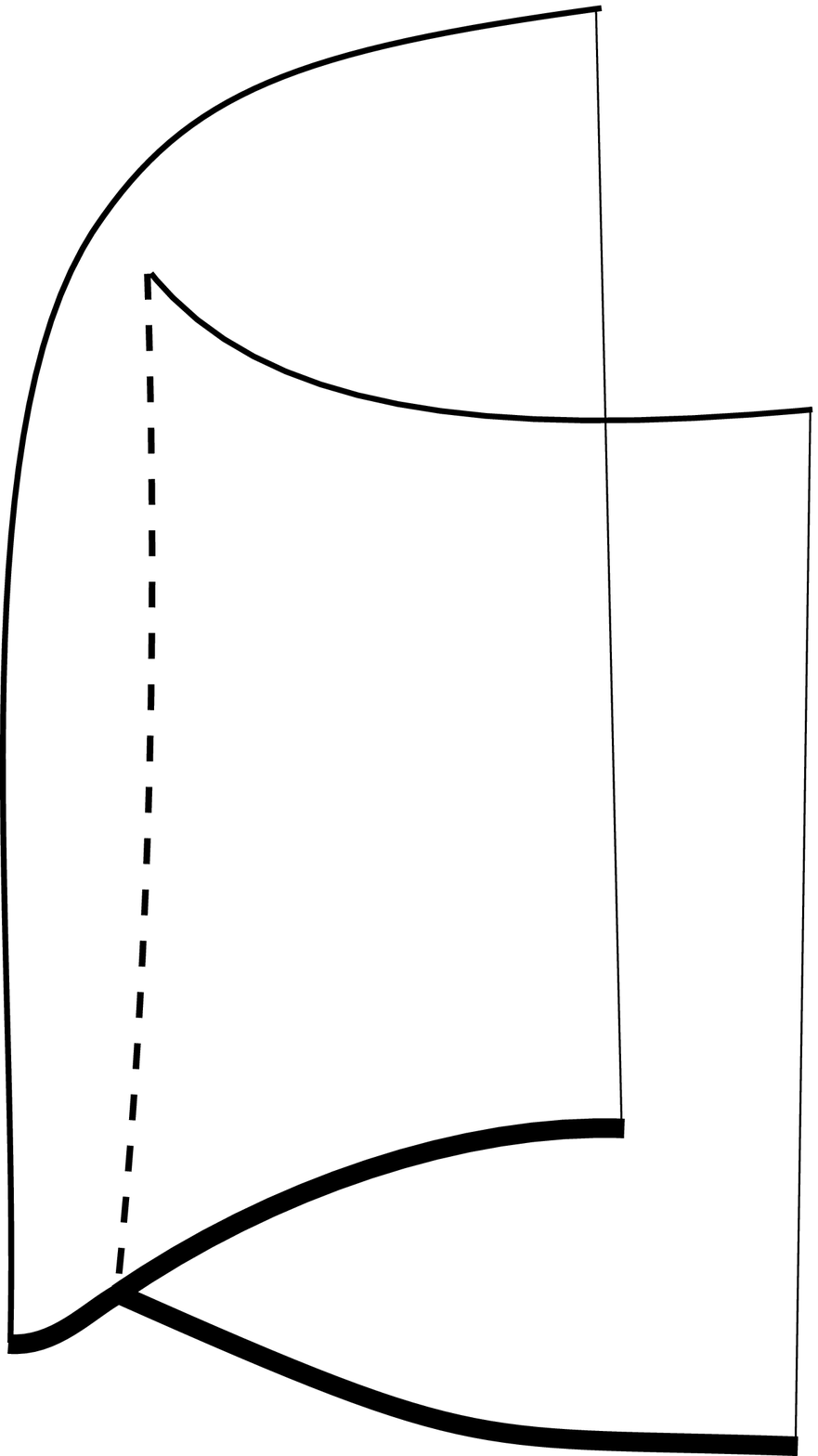, height=4cm} 
\end{center} 
\caption{}  
\label{sarkok}  
\end{figure}

\begin{lem}\label{boldlongneigh}
The preimage $q_{f_4}^{-1}(A)$ is isotopic to a regular neighborhood of $L$. 
\end{lem}
\begin{proof}
It is enough to show that 
$q_{f_4}^{-1}(A)$ is diffeomorphic to $L \x D^2$ 
extending naturally the $L \x S^1$ structure on its boundary
since
by Lemma~\ref{boldlong} the union of tori $\del q_{f_4}^{-1}(A)$ contains a longitude
isotopic to $L$.
Moreover it is enough to show that the $q_{f_4}$-preimage of the part of $A$ homeomorphic to 
the CW complex in Figure~\ref{sarkok} is diffeomorphic to
$[0,1] \x D^2$, where the $q_{f_4}$-preimage of the  two vertical edges on the right-hand side 
of the $2$-complex of Figure~\ref{sarkok}
corresponds to $\{ 0, 1 \} \x D^2$.
Clearly  the $q_{f_4}$-preimage of the  two vertical edges on the right-hand side 
is diffeomorphic  to $\{ 0, 1 \} \x D^2$ since
$q_{f_4}^{-1}(x)$ is a circle for any $x$ lying in the two vertical edges
except if $x$ is one of the two top ends.
If  $x$ is one of the two top ends, then $q_{f_4}^{-1}(x)$ is one point since
it is a definite fold singularity.
The $q_{f_4}$-preimage  of the backward sheet in Figure~\ref{sarkok}
is diffeomorphic to $[0,1] \x D^2$
minus $I \x D^2$ for an interval $I$.
The $q_{f_4}$-preimage of the forward sheet is diffeomorphic to
$I \x D^2$.
\end{proof}

\begin{cor}
Any longitudinal curve in $q_{f_4}^{-1}(\del A)$ is isotopic to $L$. 
\end{cor}

In order to obtain the map $f_5$, we modify the map
$f_4|_{q_{f_4}^{-1}(A)} \co L \x D^2 \to \R^2$ outside a neighborhood
of $q_{f_4}^{-1}(\del A)$ as it is shown by
Figure~\ref{kinovesekcuspelimfelnovelatkot}: our goal is to have the
arrangement that if for a cusp singularity $q_1 \in S^3$ the point
$q_{f_5}(q_1)$ is connected in $W_{f_5} - A'$ to $\del A$ by a
$1$-cell $\ga$ mapped into $\R^2$ parallel to $v$ and $\ga$
corresponding to an indefinite fold curve, then a definite fold curve
should connect $q_1$ to another cusp $q_2$ with the same property for
$q_{f_5}(q_2)$.  Thus we obtain a map $f_5$ such that
$q_{f_5}^{-1}(W_{f_5} - A')$ is isotopic to a regular neighborhood of
$L$ by the same argument as in Lemma~\ref{boldlongneigh}.  Also
$q_{f_5}^{-1}(W_{f_5} - A')$ coincides with $q_{f_4}^{-1}(A)$ and
$f_5$ coincides with $f_4$ in a neighborhood of $q_{f_5}^{-1}(A')$.

We arrange the cusps of $f_5$ in 
$q_{f_4}^{-1}(A)$ 
to form pairs as follows.
In $W_{f_5}$ sheets are attached to $B$ 
along arcs  of type {\sc  (b)}
(possibly containing points of type {\sc (c)} at some endpoints).
Walking along the bands $B$ and
restricting ourselves to the intersection of the sheets and $W_{f_5} - A'$,
we have that
every sheet contains a pair of cusps and
every second sheet contains a singular arc of type {\sc (a)}
connecting its pair of cusps, for example, see Figure~\ref{kinovesekcuspelimfelnovelatkot}.

\begin{figure}[ht] 
\begin{center} 
\epsfig{file=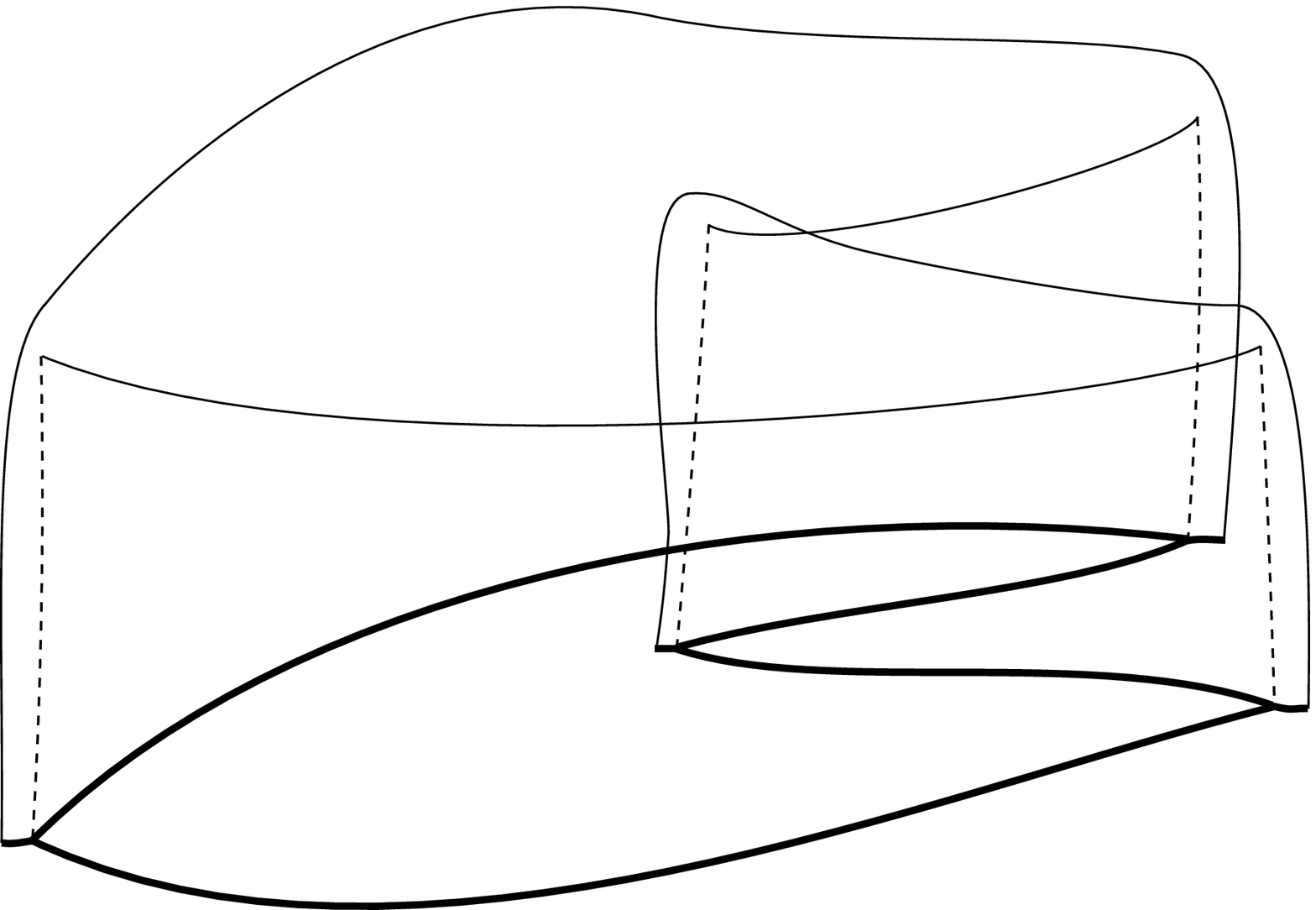, height=5cm} 
\end{center} 
\caption{{\bf The Stein factorization of $f_5|_{q_{f_5}^{-1}(W_{f_5} - A')} \co L \x D^2 \to \R^2$.}
There are two $\mathcal P$-pairs of cusps.}  
\label{kinovesekcuspelimfelnovelatkot}  
\end{figure}

A natural pairing  is that  two cusps
form a pair if they are in the same sheet and they are connected by 
a singular arc of type {\sc (a)}. We refer to this pairing as $\mathcal Q$-pairing.
We also define another pairing $\mathcal P$:
two cusps
form a $\mathcal P$-pair if they are in the same sheet and they are {\it not} connected by 
any singular arc of type {\sc (a)}.

\subsection*{Step 6}
In this step, we eliminate the cusps of $f_5$ contained in
$q_{f_5}^{-1}(W_{f_5} - A')$. These cusps are mapped by $f_5$ in the
direction of $v$ far from $\overline L$ and arranged into $\mathcal
P$-pairs in the previous step.  The restriction of the resulting map
$f_6 \co S^3 \to \R^2$ to a link isotopic to $L$ will be an
embedding. (Hence after this step the construction of the claimed map
$F$ on $M$ will be easy.)

We have exactly $|T|/2$ $\mathcal P$-pairs of cusps in $q_{f_5}^{-1}(W_{f_5} - A')$. Observe that
for each component of $L$ one
 $\mathcal P$-pair can be eliminated immediately: for example 
  in  Figure~\ref{kinovesekcuspelimfelnovelatkot}
 the pair on the ``highest'' sheet is in the sufficient 
position to eliminate. In the following, we deal with the other $\mathcal P$-pairs.

More concretely, we perform the deformations and the eliminations
of the pairs of cusps of $f_5$  in $q_{f_4}^{-1}(A)$
 as shown in Figure~\ref{cuspcsere} as follows.
 
  \begin{figure}[ht] 
\begin{center} 
\epsfig{file=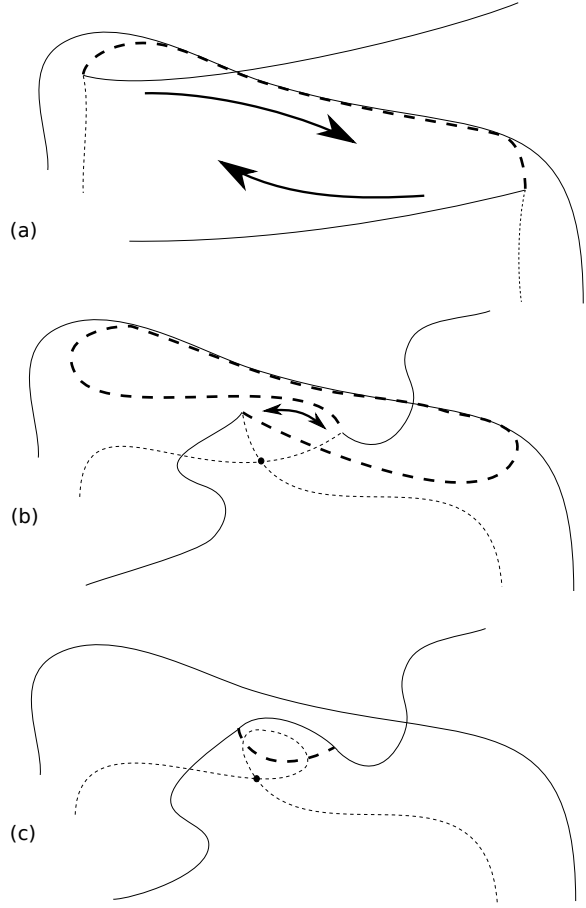, height=12cm} 
\end{center} 
\caption{{\bf Moving and eliminating the cusps.} 
We move and eliminate the $\mathcal P$-pair of cusps along the arrows.
The dashed arcs represent
$1$-complexes used to deform $\si$  in the proof of Lemma~\ref{vegsoiso}.}  
\label{cuspcsere}  
\end{figure} 

  First, by using  Lemma~\ref{cuspmovelem} we   move 
 each pair of cusps 
having the position as in Figure~\ref{cuspcsere}(a) 
to the position as in Figure~\ref{cuspcsere}(b) 
thus creating a singularity of type {\sc  (d)}. Then 
by using  Lemma~\ref{cuspelimlem}
we eliminate each pair of cusps, see
Figures~\ref{cuspcsere}(b) and \ref{cuspcsere}(c).

The resulting map will be denoted by $f_6$ (see
Figure~\ref{vegso}). Notice that $f_6$ and $f_5$ coincide in a
neighborhood of $q_{f_5}^{-1}(A')$.  The deformations above yield
definite fold curves $K \subset S^3$, whose image under $f_6$ is an
embedding into $\R^2$ as indicated in Figure~\ref{vegso} by the bold
curve.

\begin{figure}
\begin{center} 
\epsfig{file=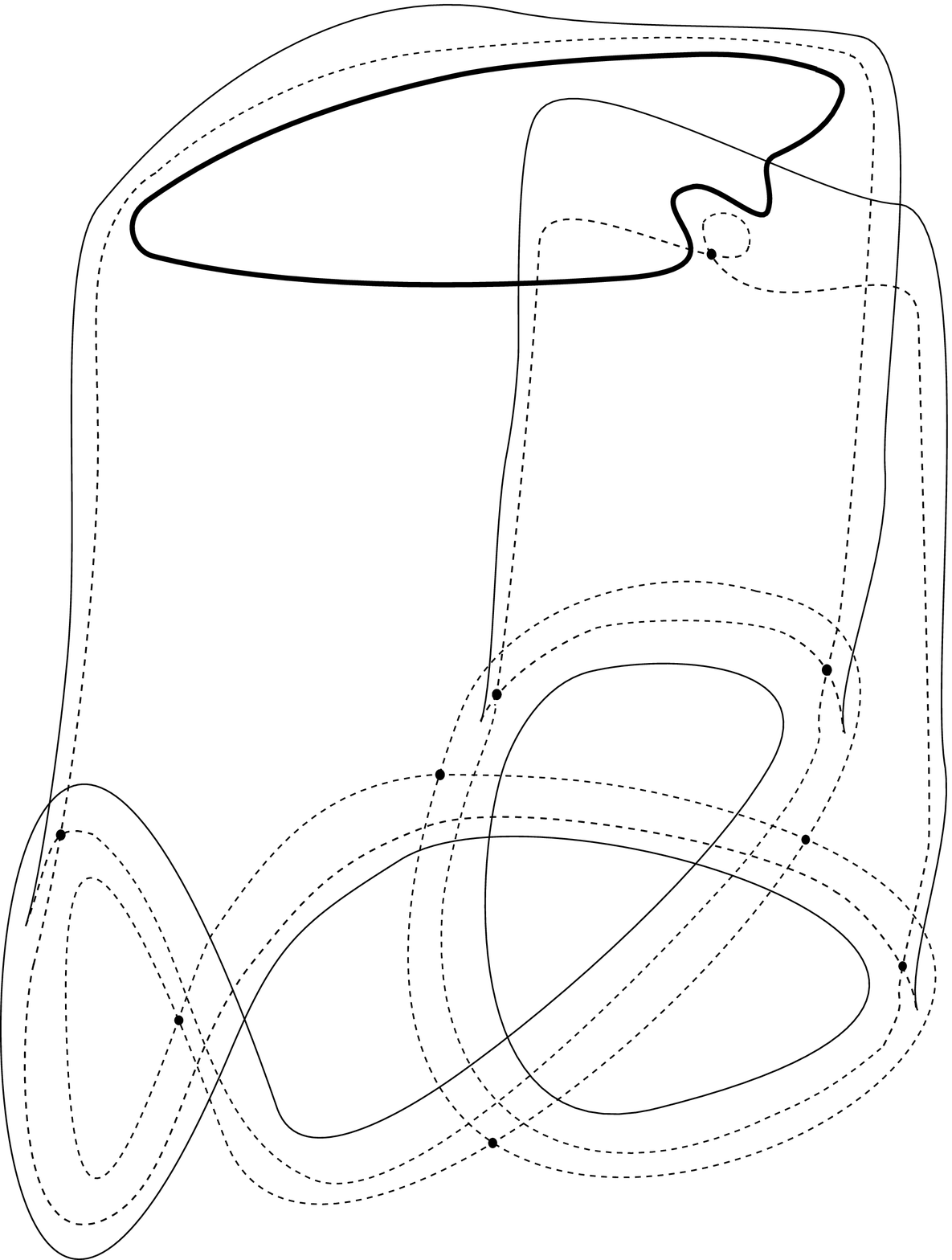, height=13cm} 
\end{center} 
\caption{{\bf The Stein factorization of the stable map $f_6 \co S^3
    \to \R^2$.} (The circle $f_6(C)$ is omitted.)}
\label{vegso}  
\end{figure}

\begin{lem}\label{vegsoiso}
The link $K$ is isotopic to $L$.
\end{lem}
\begin{proof}
By Lemma~\ref{boldlong} the link $L$ is isotopic to a longitude of the
union of tori $q_{f_4}^{-1}(\del A)$.  In Step $6$ we modify $f_5$
only inside $q_{f_4}^{-1}(A)$.  The subcomplex $\si$ of $\del A$ used
in the proof of Lemma~\ref{boldlong} is PL-isotopic to a
$1$-dimensional PL submanifold $\si'$ of $W_{f_5} - A'$ such that
$\si'$ goes through the singular curves of type {\sc (a)} appearing in
the $\mathcal Q$-pairing at the end of Step 5 and goes through the top
of $W_{f_5} - A'$, i.e.\ the top of the $2$-complex in
Figure~\ref{kinovesekcuspelimfelnovelatkot}.  To be more precise, in
Figure~\ref{cuspcsere}(a) the part of $\si'$ connecting the two cusp
endpoints of the singular arcs of type {\sc (a)} is represented by a
bold dashed arc and denoted by $\si''$.  During the moving of the pair
of cusps as depicted by the arrows in Figure~\ref{cuspcsere}(a),
$\si''$ is deformed to the curve $\si'''$ represented by a bold dashed
arc in Figure~\ref{cuspcsere}(b).  This deformation gives an isotopy
between some liftings to $S^3$ of $\si''$ and $\si'''$.  Since a part
of $\si''$ is collinear to a singular arc of type {\sc (a)} as we can
see in Figure~\ref{cuspcsere}(a), any lifting to $S^3$ of $\si''$ is
isotopic to any other lifting.  Hence further deforming $\si'''$ to
$\si''''$ represented by the bold dashed curve in
Figure~\ref{cuspcsere}(c) yields an isotopy between some liftings of
$\si''$ and $\si''''$.  Finally, changing again the lifting to $S^3$
of $\si''''$ if necessary, we eliminate the pair of cusps as indicated
in Figure~\ref{cuspcsere}(b) and deform $\si''''$ to be identical to
the type {\sc (a)} singular arc appearing at the elimination in
Figure~\ref{cuspcsere}(c).  All this process gives an isotopy in $S^3$
between $K$ and a lifting of $\si$, hence an isotopy between $K$ and
$L$.
\end{proof}

\subsection*{Step 7}
As a final step, we perform the given surgeries along $K$ with the
appropriate coefficients.  Since $f_6|_K$ is an embedding into $\R^2$
on each component of $K$, and $K$ consists of definite fold singular
curves such that the local image of a small neighborhood of the
definite fold curve is situated ``outside'' of the image of the
definite fold curve, a map of $M$ is particularly easy to construct: a
small tubular neighborhood $N_K$ of $K$, which is diffeomorphic to $K
\x D^2$, is glued back to $S^3- {\rm {int}}\, N_K$ such that $\{ pt.\}
\x \partial D^2$ maps to a longitude in $\del(M- {\rm {int}}\, N_K)$, hence $N_K$
can be mapped into $\R^2$ as the projection $\pi \co K \x D^2 \to
D^2$.  This $\pi$ extends over $M- {\rm {int}}\, N_K$ and the resulting map $M \to
\R^2$ is stable. Let us denote it by $F$.
  
It is easy to see that $F$ has the claimed 
properties:
\subsubsection*{The Stein factorization $W_F$ is homotopy equivalent 
to the bouquet $\bigvee_{i=1}^{{\rm {n}}( L)} S^2$:}
The Stein factorization $W_{f_4}$ is clearly 
contractible. The CW-complexes $W_{f_5}$ and $W_{f_6}$ are still contractible since the corresponding
steps do not change the homotopy type. At the final surgery
we attach  a $2$-disk to $W_{f_6}$
for each component of $L$.

 \subsubsection*{
The number of cusps of $F$ is equal to ${\rm t}_v( \overline L )$:} Each
 point in $f_1(L')$ at which $f_1(L')$ is tangent to the chosen
 general position vector $v$ (these are exactly the points of the set
 $\bar f_1(T))$ corresponds to a cusp of $F$ by the construction and
 there are no other cusps. $|T| = {\rm t}_v( \overline L )$ hence we get
 the statement.

\subsubsection*{All the non-simple singularities of $F$ are
 of type {\rm \sc (d)}:}
 This follows from the fact  that singularities of type {\sc  (e)} never appear 
 during the construction.
 
 \subsubsection*{The number of the 
 non-simple singularities of $F$
   is equal to 
${\rm cr}(  \overline L ) + \frac{3}{2}{\rm t}_v(  \overline L ) -{\rm n}(L)$:}
Each crossing of the diagram $\overline L$ gives a singularity of type {\sc (d)}.
Also each point in $T$ gives a singularity of type {\sc (d)} by the construction.
Finally, the movement illustrated in Figure~\ref{cuspcsere}(b)
gives one singular point of type {\sc (d)} for each pair of points in $T$ except one pair
for each component of $L$.

\subsubsection*{The number 
of non-simple singularities which are not connected by any  singular arc of type {\sc (b)}
to any cusp is equal to ${\rm cr}(  \overline L ) + \frac{1}{2}{\rm t}_v(  \overline L ) -{\rm n}(L)$:}

In the previous argument, if we do not count the singularities of type {\sc (d)} corresponding to
the $v$-tangencies of $f_1(L')$, then we get the statement.

\subsubsection*{The number of simple singularity crossings of $F$ in $\R^2$ is no more than
$8{\rm cr}(  \overline L ) + 6\ell(\overline L, v){\rm t}_v(  \overline L ) + {\rm t}_v(  \overline L )^2$:}
We can suppose that the number of simple singularity crossings of $f_4|_{q_{f_4}^{-1}(A'))}$
is no more than 
$8{\rm cr}(  \overline L ) + 2{\rm t}_v(  \overline L ) + 6\ell(\overline L, v){\rm t}_v(  \overline L )$.
The maps $f_4$, $f_5$, $f_6$ and $F$ coincide in a neighborhood of
$q_{f_4}^{-1}(A')$ and also their images coincide 
in the half plane bounded by the line $l$ and lying in the direction $-v$ (for the notations, see Step 5). 
The simple singularity crossings of $F$
in $F(q_{f_4}^{-1}(A))$ come from the intersections of the $\bar F$-images of the
``sheets'' attached to the bands $B \subset W_F
$ (for the notation, see Step 2). For example, in Figure~\ref{vegso},
two such  sheets intersect on the right-hand side  in four simple singularity crossings.
Hence we obtain an upper bound for the 
number of simple singularity crossings of $F$ in $F(q_{f_4}^{-1}(A))$ 
if we suppose that all the sheets intersect each other in eight crossings.
This gives the upper bound 
$$8\left( \frac{{\rm t}_v(  \overline L )}{2} -1 +  \frac{{\rm t}_v(  \overline L )}{2} -2 + \cdots +1 \right )=
4\frac{{\rm t}_v(  \overline L )}{2} \left (\frac{{\rm t}_v(  \overline L )}{2} -1 \right ) =
 {\rm t}_v(  \overline L )^2 - 2{\rm t}_v(  \overline L ).$$
Thus we obtain the upper bound 
$$8{\rm cr}(  \overline L ) + 2{\rm t}_v(  \overline L ) + 6\ell(\overline L, v){\rm t}_v(  \overline L ) +  {\rm t}_v(  \overline L )^2 - 2{\rm t}_v(  \overline L ) = 
8{\rm cr}(  \overline L ) + 6\ell(\overline L, v){\rm t}_v(  \overline L ) +  {\rm t}_v(  \overline L )^2$$
for all the simple singularity crossings of $F$.

 \subsubsection*{
The number of connected components of the singular set of $F$ is no more than ${\rm {n}}( L) + \frac{3}{2}{\rm t}_v(\overline L) +1$:}
The curve $C$ is a component and
the links $L$ and $L'$ give singular set components as well. Also 
the cusp elimination in Step 3 gives additional ${\rm t}_v(\overline L)$ components.
Steps 4 and 5 clearly do not increase more the number of singular set components.
In Step 6 the changings showed in Figure~\ref{cuspcsere}  increase 
the number of components by at most $\frac{1}{2}{\rm t}_v(\overline L)$. 
Finally Step 7 decreases it by ${\rm {n}}( L)$.

 \subsubsection*{
The maximal number of the connected components of any fiber of $F$ is no 
more than ${\rm t}_v(  \overline L ) + 3$:}
The maximal number of the connected components of any fiber of $f_1$
is $3$. This value is no more than $3+{\rm t}_v(  \overline L )$ for $f_2$, \ldots, $f_5$
and also for $f_6$. When we perform the surgery in Step 7,
$3+{\rm t}_v(  \overline L )$ is still an upper bound hence we get the statement.

 \subsubsection*{
The indefinite fold singular set of $F$:}
Finally the statement of (8) about the indefinite fold singular set of $F$   
 is obvious from the construction. This finishes the proof of Theorem~\ref{fothm}.
\end{proof}

\begin{rem}
Suppose we have two links in $S^3$.
If the projections of the two links coincide, then the resulting 
stable maps 
on the two $3$-manifolds
in the construction described above  will have the same Stein factorizations.
Therefore only the Stein factorization itself is a very week invariant of the $3$-manifold.\footnote{The paper \cite{Sa95} is closely related to this remark.}
\end{rem}

\begin{proof}[Proof of Theorem~\ref{foldcsak(e)}]
Let $M$ be a closed orientable $3$-manifold obtained by an integral surgery along a link in $S^3$.
Theorem~\ref{fothm} gives a stable map $F$ of $M$ into $\R^2$ without singularities of type {\sc  (e)}.
We can eliminate the cusps of $F$ without introducing any singularities of type {\sc  (e)}.
Indeed, the map constructed by Theorem~\ref{fothm} has an even number of cusps,
whose $q_F$-image is situated in $B \subset W_F$. 
Moreover since the locations of the $F$-images of the cusps
are at the $v$-tangencies of $\overline L$, 
each cusp $c$ has a pair $c'$
which can be moved close to $c$ (thus possibly creating  new singular points
of type {\sc (d)})  and can be used to eliminate these pairs
in the sense of Lemmas~\ref{cuspmovelem} and \ref{cuspelimlem}.
\end{proof}

\begin{rem}
By \cite{ElMi} every closed orientable $3$-manifold
has a wrinkled map into $\R^2$ since any orientable $3$-manifold
is parallelizable. This argument  leads to another proof of 
Theorem~\ref{foldcsak(e)}. However, the $h$-principle used  in the proof of the results in 
\cite{ElMi}
does not provide any construction for the wrinkled map.
\end{rem}

Next we give the proof of the estimate given in \eqref{eq:egysz} in Section~\ref{sec:intro}.
\begin{lem}\label{egyszerusit}
$\ell(\overline L, v) \leq {\rm t}_v(  \overline L ) -1$.
\end{lem}
\begin{proof}
For any $v$-tangency $p$ we have
$\ell(\overline L, v, p) \leq {\rm t}_v(  \overline L ) -1$ since
by going along the components of $L$ in the diagram $\overline L$,
 in order to pass through the intersections of the half line
 emanating from $p$ in the direction of $v$,
 for each intersection
 one needs to pass through a $v$-tangency as well.
 \end{proof}

\subsection{Estimates for $TB^-$}

Recall that the Thurston-Bennequin number ${\rm tb}(\Legknot )$ of a
Legendrian knot $\Legknot $ can be computed 
through the simple formula
$$
{\rm tb}(\Legknot )=w(\Frontproj )-\frac{1}{2}(\# cusps (\Frontproj)).
$$

\begin{proof}[Proof of Theorem~\ref{nagyTBbecsles}]
By Theorem~\ref{fothm} (5) and Lemma~\ref{egyszerusit} 
we have 
\[
{\rm s}(F) \leq 8{\rm cr}( \overline L ) + 7{\rm {t}}_v(\overline L)^2 - 6{\rm
  {t}}_v(\overline L)\] for the constructed stable map $F$. (Here, again,
$\overline L$ denotes the generic projection of the knot $L$ we get from
the front projection of the Legendrianization $\Legknot$ of $L$ by
rounding the cusps.)  Since ${\rm d}(F) = {\rm s}(F) + {\rm ns}(F)$,
by Theorem~\ref{fothm} (3), (5) and Lemma~\ref{egyszerusit} we have
$${\rm d}(F) \leq 9{\rm cr}( \overline L ) + 7{\rm {t}}_v(\overline L)^2 -
\frac{9}{2}{\rm {t}}_v(\overline L) -{\rm n}(L).$$

If $\Frontproj$ has only negative crossings, then
the Thurston-Bennequin number ${\rm tb}(\Legknot )$ is equal to
$- {\rm cr}(  \overline L ) - \frac{1}{2} {\rm {t}}_v(\overline L)$, where
$v$ is the vector in which the front projection has no tangency.

Hence
$$28{\rm tb}(\Legknot )^2 = 28{\rm cr}( \overline L ) ^2 + 28{\rm cr}( \overline L
) {\rm {t}}_v(\overline L) + 7{\rm {t}}_v(\overline L)^2$$ and
$$28{\rm cr}( \overline L ) ^2 + 28{\rm cr}( \overline L ) {\rm {t}}_v(\overline L) +
7{\rm {t}}_v(\overline L)^2 \geq 9{\rm cr}( \overline L ) + 7{\rm {t}}_v(\overline
L)^2 - \frac{9}{2}{\rm {t}}_v(\overline L) -{\rm n}(L).$$

Thus $|{\rm tb}(\Legknot )| \geq \frac{\sqrt{{\rm d}(F)}}{\sqrt{28}}$,
implying (by the fact that ${\rm tb}(\Legknot )$ is negative for a
knot admitting a projection with only negative crossings)
\begin{equation}\label{dequ}
{\rm tb}(\Legknot ) \leq -\frac{\sqrt{{\rm d}(F)}}{\sqrt{28}}.
\end{equation}

Also by Theorem~\ref{fothm}  (4), we have
$$|{\rm tb}(\Legknot)| = {\rm cr}( \overline L ) + \frac{1}{2} {\rm
  {t}}_v(\overline L) \geq {\rm nsnc}( F ) + 1,$$ which gives
\begin{equation}\label{nsncequ}
{\rm tb}(\Legknot ) \leq - {\rm nsnc}( F ) - 1.
\end{equation}

Finally note that ${\rm d}(F) \geq {\rm s}(F)$ for any stable map $F$,
and by taking the minimum for all the stable maps in (\ref{dequ}) and
(\ref{nsncequ}), we get the statement.
\end{proof}

\end{document}